\documentclass[12pt,a4paper]{article}
\usepackage[utf8]{inputenc}
\usepackage[english]{babel}
\usepackage{amsmath}
\usepackage{amsfonts}
\usepackage{cases}
\usepackage{graphicx}
\usepackage{fouridx}
\usepackage{amsthm}
\usepackage{amssymb,amscd}
\usepackage{enumerate}
\usepackage{mathtools}
\usepackage{tikz-cd}
\usepackage[all, knot]{xy}
\usepackage{multirow} 
\usepackage{hyperref}
\usepackage{bm}
\usepackage{multicol}
\usepackage{multirow}
\author{Nejib Saadaoui
	\thanks{Nejib Saadaoui. 
		Université de Gabès, 
		Institut Supérieur d'Informatique et de Multimédia de  Gabès,
		Campus universitaire-BP 122, 6033 Cité El Amel 4, Gabès, TUNISIE.		
		, e-mail:\textbf{najibsaadaoui@yahoo.fr}}}
\title{Split extensions of BiHom-Lie algebras }
\newtheorem{theorem}{Theorem}[section]

\newtheorem{thm}[theorem]{Theorem}
\newtheorem{lem}[theorem]{Lemma}

\newtheorem{defn}[theorem]{Definition}

\numberwithin{equation}{section}
\usepackage{multicol}
\input{epsf.sty}

\usepackage{varioref}
\newcommand{\enstq}[2]{\left\{#1\mathrel{}\middle /\mathrel{}#2\right\}}

\newcommand{\K}{\mathbb{K}}

\begin{document} 
	\maketitle
	
	\begin{abstract}
In this paper, we introduce the notion of split extension  of  BiHom-Lie algebra
and construct the corresponding cohomology. Also,   
we establish a one-to-one correspondence between the equivalence classes 
of  extensions of a
BiHom-Lie algebra $L$ by an abelian BiHom-Lie algebra $V$ and 
 it's second cohomology group.		
	\end{abstract}
\section*{Introduction}
Hom-Lie algebras were originally introduced in \cite{HLS} is an anti-commutative algebra satisfying a Jacobi identity twisted
by a linear map. When this linear map is the identity, the deﬁnition of Lie algebra is recovered.


In \cite{GMMP}, Graziani et al. 
defined  BiHom–Lie algebras.
that is a generalized Hom-Lie algebra endowed with two commuting
multiplicative linear maps. As in the Hom-Lie case, if the two linear maps are equal to the identity we find the Lie algebra.\\

In Lie algebras, the second cohomology group  represents central extensions \cite{M,DBT}.  
The aim of this paper 
is to define the extension of BiHom-lie algebra and generalize the provious result    to BiHom-Li case .

This paper is organized as follows. In Section 1 we recall some preliminaries on BiHom-Lie algebras. In section 2 we introduce the notion of extension of BiHom-lie algebra and construct the   representation and 2-cocycle  corresponding  to  split extension of BiHom-Lie algebras. In section 3 we show that any   split extension $(M,d',\alpha_{M},\beta_{M})$ of $(L,\delta,\alpha)$ by $ (V,\mu, \alpha_{V},\beta_{V}) $ is equivalent to the extension $(L\oplus V,d,\alpha+\alpha_{V},\beta+\beta_{V})$. Finally, in section 4
we study the abelian split 
extensions of BiHom-Lie algebras and their relations with the second cohomology group.

Throughout the paper we will consider vector spaces over a ﬁeld $ \K $.

	\section{BiHom-Lie algebras}
%
%
In this section we recall  and introduce some definitions about BiHom-Lie  algebras  that we will need in the remainder of the paper.	
\begin{defn}\cite{GMMP}
	A BiHom-Lie algebra over a field $ \K $ is a $ 4 $-tuple $ \left( L,[\cdot,\cdot],\alpha,\beta \right)  $, where 
	$  L $ is
	a  $ \K $-linear space, $ [\cdot,\cdot] \colon L\times L\to L$ a bilinear map and $ \alpha,\, \beta\colon L\to L $
	linear mappings satisfying the following identities:
	\begin{align}
		&\alpha\circ \beta=\beta\circ \alpha, \\
		&[\beta(x),\alpha(y)]=-[\beta(y),\alpha(x)]    \text{  (skew-symmetry),}\\
		&\left[ \beta^{2}(x),[\beta(y),\alpha(z)]\right] +\left[\beta^{2}(y), [\beta(z),\alpha(x)]\right] + \left[ \beta^{2}(z),[\beta(x),\alpha(y)]\right] =0 \\
		&\qquad\qquad\text{  (BiHom-Jacobi condition)          .}   \nonumber       
	\end{align}		
	A BiHom-Lie algebra is called a multiplicative BiHom-Lie algebra if $ \alpha $ and $ \beta $ are algebraic
	morphisms, i.e. 
\[ \alpha\left( [x,y] \right) =[\alpha(x),\alpha(y)]  \quad  \text{  and   }   \quad   \beta\left( [x,y] \right) =[\beta(x),\beta(y)]   \]	
for any $ x,\,y,\, z\in L $.\\	
A homomorphism
$f\colon \left( L,[\cdot,\cdot],\alpha,\beta \right) \to \left( L',[\cdot,\cdot]',\alpha',\beta' \right) L$
 is said to be a morphism of BiHom-Lie algebras if $  f\circ \alpha =\alpha' \circ f$,  $f\circ\beta =\beta'\circ f$ and 
\begin{gather*} 
	 f([x,y]) =[f(x),f(y)]' \quad \forall x,\, y\in L .                     
\end{gather*}	
\end{defn}
\begin{defn}
 An abelian BiHom-Lie algebra $ \left(V, \alpha,\beta\right)  $ is a $\K$-vector space $V$ endowed with trivial bracket and any linear map $ \alpha\colon V\to V $ and  $ \beta\colon V\to V $.	
\end{defn}	
\begin{defn}
Let $\left(M,d,\beta,\alpha \right)   $ be a  BiHom-Lie algebra. 
The set of $2$-cochains on $M$, which we denote by $ C^2(M) $, is the set of  bilinear  maps $f$  from
$ M^2 $
to $ M$ satisfies \[f\left( \beta(a),\alpha(b)\right) =-f\left( \beta(b),\alpha(a)\right),\quad \forall a,\,b\in M    \].

A $ 1 $-cochain on $ M $  is defined to be a linear map $ f\in Hom(M) $ satisfies $f\circ \alpha =\alpha\circ f$ and $f\circ \beta =\beta\circ f$.
  Denote by $ C^1(M) $
  the set of $ k $-BiHom-cochains:
 \[ C_{B}^1(M)= \enstq{   f\in Hom(M)}{\alpha\circ f=f\circ \alpha,\beta\circ f=f\circ \beta  }       \]
\end{defn}
\begin{defn}
A sub-vector space $H$ of $M$  is a BiHom subalgebra of $ \left( M,d,\alpha,\beta \right)  $ if  $\alpha(H)\subset H$,
$\beta(H)\subset H$, and $ [H,H]\subset H $.
\end{defn}
\begin{defn}
A sub-vector space $I$ of $M$ is an ideal of $ \left( M,d,\alpha,\beta \right)  $ if 
$\alpha(I)\subset I$,
$\beta(I)\subset I$, $d(L,I)\subset I$ and $ d(I,L) \subset I$.
\end{defn}
\section{Extensions and representations of BiHom-Lie algebras}
Extensions of Lie algebras are studied in several papers, see  for example \cite{DBT,FP}. 
In this section, we will concentrate our efforts on the development of the general theory of  extensions for BiHom-Lie algebras, and we will 
we give a purely cohomological interpretation of abelian split  extensions.

\begin{defn}
	Let $ ( L,\delta,\alpha,\beta)    $, $ (V,\mu,\alpha_{V},\beta_{V}) $, and  $ ( M,d_{M},\alpha_{M},\beta_{M})     $ be   BiHom-Lie algebras
	and $ i\colon V \to M $, $ \pi \colon M\to L $
	be morphisms of  BiHom-Lie algebras. The following sequence of
	BiHom-Lie algebra 
	$$0\longrightarrow (V,\mu,\alpha_{V},\beta_{V})\stackrel{i}{\longrightarrow} (M,d_{M},\alpha_{M},\beta_{M})\stackrel{\pi}{\longrightarrow }(L,\delta,\alpha,\beta) \longrightarrow 0,$$
	is a short exact sequence if $ Im(i)=\ker (\pi) $, $ i $ is injective and $\pi$ is surjective			
	In this case, we call $ ( M,d_{M},\alpha_{M},\beta_{M})     $
	an extension of $ ( L,\delta,\alpha,\beta)    $ by $ (V,\mu,\alpha_{V},\beta_{V}) $.
\end{defn}	
\begin{defn}
	Two extensions 
	\begin{gather*}
		\xymatrix {\relax 0 \ar[r] &( V_{1},\mu_{1},\alpha_{V_{1}},\beta_{V_{1}} )   \ar[d]_{\varphi}\relax\ar[r]^-{i_{1}} \relax & (M_{1} ,d_{1},\alpha_{V_{1}},\beta_{V_{1}}) \ar[d]_{\Phi} \ar[r]^-{\pi_{1}}& (L_{1},\delta_{1},\alpha_{1},\beta_{1})\ar[r] \ar[d]_{s}&0 \\
			\relax	0\ar[r] & ( V_{2},\mu_{2},\alpha_{V_{2}},\beta_{V_{2}} ) \ar[r]^-{i_{2}}     & (M_{2},d_{2},\alpha_{M_{2}},\beta_{M_{2}}) \ar[r]^-{\pi_{2}} &(L_{2},\delta_{2},\alpha_{2},\beta_{2})\ar[r]&0}
	\end{gather*}	
	are equivalent if there is an isomorphism of BiHom-Lie algebra \[\Phi:(M_{1},d_1,\alpha_{1},\beta_{1})\rightarrow (M_{2},d_2,\alpha_{M_{2}},\beta_{M_{2}})\] such that $\Phi \circ \ i_{1}= i_{2}\circ \varphi$ and $\pi_{2}\ \circ\ \Phi=s\circ \pi_{1}.$
\end{defn}

\begin{defn}	
	An extension  \[0\longrightarrow (V,\mu,\alpha_{V},\beta_{V})\stackrel{i}{\longrightarrow} (M,d,\alpha_{M},\beta_{M})\stackrel{\pi}{\longrightarrow }(L,\delta,\alpha,\beta) \longrightarrow 0\] is called:
	\begin{enumerate}[(1)]
		\item trivial if  there exists an ideal $ I $ complementary to
		$\ker \pi$, 
		\item	split if there exists a BiHom-subalgebra $ S\subset M$
		complementary to $ \ker \pi $,
		\item central if the $ \ker \pi $ is contained in the center $ Z(M) $ of $ L $. That is $ d(i(V),M)=0. $
		\item  abelian if $ \mu=0 $. In this case we say $ ( M,d,\alpha_{M},\beta_{M})     $ 
		an extension of $ ( L,\delta,\alpha,\beta)    $ by $ (V,\alpha_{V},\beta_{V}) $. 	
	\end{enumerate}

\end{defn}	
Let  $ (L,\delta,\alpha,\beta) $ be a BiHom-Lie algebra   and let  $ (V,\alpha_{V},\beta_{V})$ be an abelian BiHom-Lie algebra. Let 

$0\longrightarrow (V,\alpha_{V},\beta_{V})\stackrel{i_{0}}{\longrightarrow} (L\oplus V,d,\alpha+\alpha_V,\beta+\beta_{V})\stackrel{\pi_{0}}{\longrightarrow }(L,\delta,\alpha,\beta) \longrightarrow 0$ be a short exact sequence, where 
$ i_{0}(v)=v $ and $ \pi_{0}(x+v)=x $ for all $x,y\in L, u,v\in V$.
We want to determine the conditions under which $ \left( L\oplus V,d,\alpha+\alpha_V,\beta+\beta_{V}\right)  $ is a split extension of $ ( L,\delta,\alpha,\beta)    $ by $ (V,\mu,\alpha_{V},\beta_{V}) $.
%
%
%
%
%
%
%
For convenience, we introduce the following notation for certain subspaces  of cochains  on $ M=L\oplus V $:
\begin{gather*}
	 C^{2}(LV,V)=\enstq{   f\in Hom(M^2,M)}{f\colon L\times V\to V }, \\ 
	 C^{2}(VL,V)=\enstq{   f\in Hom(M^2,M)}{f\colon V\times L\to V },\\
	 C^{2}(L^2,V)=\enstq{   f\in Hom(M^2,M)}{f\colon L\times L\to V }.                                     
\end{gather*}

For any $ f, g\in  C^2(M) $, 
we define $f \circ g\in  Hom(M^{3},M) $ by
\begin{gather*}
f \circ g	(a,b,c)\\=f(\beta_{M}^2(a),f(\beta_{M}(b),\alpha_{M}(c)))
	+f(\beta_{M}^2(b),g(\beta_{M}(c),\alpha(a)))
	+f(\beta_{M}^2(c),g(\beta_{M}(a),\alpha_{M}(b)))
\end{gather*}
 and  
$ [f,g]\in   Hom(M^{3},M) $ by 
\begin{equation}\label{bracketNR1}
	[f,g]=f \circ g +g\circ f.
\end{equation}
Then, if $f=g$ , we have
\begin{gather*}
[f,f](a,b,c)=	2f \circ f	(a,b,c)\\=2f(\beta^2(a),f(\beta(b),\alpha(c)))
	+2f(\beta^2(b),f(\beta(c),\alpha(a)))
	+2f(\beta^2(a),f(\beta(b),\alpha(c))).
\end{gather*}
Hence,
\begin{gather*}
\Big(\left(M,f,\alpha_{M},\beta_{M} \right) \text{  BiHom-Lie algebra      }\Big) \iff \Big([f,f]=0\Big).	
\end{gather*}
Thus, $[d,d]=0$.

Since  $ V=\ker(\pi_{0}) $, then $ V $ is an ideal of BiHom-Lie algebra $L\oplus V  $ and $d\in C^2(M)$, then $ [d,d]=0 $ and there exists       $\lambda_{l}\in C^{2}(L V,V)$, $\lambda_{r}\in C^{2}(V L,V)$, $\theta\in C^{2}(L^2,V) $   
such that 
$d=\delta+\lambda_{r}+\lambda_{l}+\theta $. Hence
\begin{gather*}
	[\delta+\lambda_{r}+\lambda_{l}+\theta,\delta+\lambda_{r}+\lambda_{l}+\theta ] =0                                             
\end{gather*}
Therefore
\begin{gather*}
	\underbrace{	[\delta,\delta ]}_{=0}+\underbrace{2	[\delta,\lambda_{r}+\lambda_{l} ]+[\lambda_{r}+\lambda_{l},\lambda_{r}+\lambda_{l} ]}_{\in C^{2,1}}
	+\underbrace{	[\delta+\lambda_{r}+\lambda_{l}, \theta]}_{\in C^{3,0}}
 =0.                                             
\end{gather*}
Then we can deduce that
\begin{align}
	&	2[\delta,\lambda_{r}+\lambda_{l} ]+[\lambda_{r}+\lambda_{l},\lambda_{r}+\lambda_{l} ]:\text{ The "module" or "representation" relation          }\label{repOCT}\\
	&[\delta+\lambda_{r}+\lambda_{l}, \theta]=0: \theta  \text{ is a $2$-cocycle with values in   } V.\label{cocycle2}
\end{align}	
By  $ d $ is a $ 2 $- cochain on $M$, we have
 \begin{gather}
 	d\left( \beta(x)+\beta_{V}(v),\alpha(y)+\alpha_{V}(w)\right) =-d\left( \beta(y)+\beta_{V}(w),\alpha(x)+\alpha_{V}(v)\right)\label{cocyDec} 
\end{gather}
 for all $x,y\in L$ and $v,w\in V$.
 
 The equalities \eqref{repOCT}, 
 and \eqref{cocyDec} make the following definition:
 \begin{defn}
 	A representation (module ) of the BiHom-Lie algebra $ \left(L,\delta,\alpha,\beta \right)  $  is a Hom-module $ \left(V,\alpha_V,\beta_V \right)  $    equiped with two $ L $-actions (left and right) $ \lambda_l\colon L\times V\to V $,  $\lambda_r\colon V\times L\to V  $ and a bilinear map $ \theta\in C^{2,0}(L^2,V) $ satisfying the following axioms,
 	\begin{align*}
 		\lambda_{r}\left(\beta_V(v),\alpha(y) \right)&=- \lambda_{l}\left(\beta(y),\alpha(v) \right)\\	
 		\lambda_{r}\left(\beta^{2}(v),\delta(\beta(x),\alpha(y)) \right)&=\lambda_{l}\left(\beta^{2}(x),\lambda_{r}(\beta(v),\alpha(y)) \right)+\lambda_{l}\left(\beta^{2}(y),\lambda_{r}(\beta(x),\alpha(v)) \right)  	                              
 	\end{align*}
 	for any $ x,y\in L$ and $v\in V $.  We say that   $ \left(V,\lambda_{r}+\lambda_{l},\alpha_V,\beta_V \right)  $ is a representation of $L$ by $V$.            
 \end{defn}
Now, by \eqref{cocycle2} and \eqref{cocyDec} we define the $2$ cocycle as follow.
\begin{defn}
	Let $ \left(L,\delta,\alpha,\beta \right)  $ be a  BiHom-Lie algebra and $ \left(V,\lambda_{r}+\lambda_{l},\alpha_V,\beta_{V} \right)  $ be a representation of $ L $. If the bilinear map 
	$ \theta $ satisfies  
	\begin{gather*}
			\theta(\beta(x),\alpha(y))=-\theta\left( \beta(y),\alpha(x)\right)  \text{  and   }\\ 
		\theta\left(\beta^{2}(x),\delta(\beta(y),\alpha(z)) \right)
		-\theta\left(\beta^{2}(y),\delta(\beta(x),\alpha(z)) \right)
		+\theta\left(\beta^{2}(z),\delta(\beta(x),\alpha(y)) \right)\\
		+\lambda_{l}\left(\beta^{2}(x),\theta(\beta(y),\alpha(z)) \right)
		-\lambda_{l}\left(\beta^{2}(y),\theta(\beta(x),\alpha(z)) \right)
		+\lambda_{l}\left(\beta^{2}(z),\theta(\beta(x),\alpha(y)) \right)=0
	\end{gather*}	
	Then $ \theta $ is called a $ 2 $- cocycle  of the BiHom-Lie algebra $ \left(L,\delta,\alpha,\beta \right)  $ related to the representation $ \left(V,\lambda_{r}+\lambda_{l},\alpha_V,\beta_{V} \right)  $, or simply a $ 2$-cocycle  of $ L $ on $V$. The set of all $ 2 $-cocycles on $ V $ is denoted $ Z^{2}(L,V,\lambda_{r}+\lambda_{l}) $.	                              
\end{defn}
\begin{defn}
		Let $ ( V,\lambda_{r}+\lambda_{l} ,\alpha_V,\beta_{V}) $ be a representation  of  a  $ L $ by $ V $ and $d=\delta+\lambda_{r}+\lambda_{l}$.
We call $2$-coboundary operator of BiHom-Lie algebra $L$ the  map 
$ D_{2}\colon C^1(L,V)\to C^2(L,V) $
defined by

\begin{equation*}
	  D_{2}(f)=[d,f]                                              
\end{equation*}
 where $[\cdot,\cdot]$ is the bilinear map defined in \eqref{bracketNR1}. Then we have 
 \begin{equation*}
 Z^{2}(L,V,\lambda_{r}+\lambda_{l})=\ker( D_{2}).
\end{equation*}
\end{defn}
We summarize the main facts in the theorem below.
\begin{theorem}
	\[0\longrightarrow \left(V,\alpha_V,\beta_{V} \right)\stackrel{i_{0}}{\longrightarrow} (L\oplus V,d,\alpha+\alpha_V,\beta+\beta_{V})\stackrel{\pi_{0}}{\longrightarrow }(L,\delta,\alpha,\beta) \longrightarrow 0\] is an   extension of $L$ by $V$ 
	if and only if $ d=\delta+\lambda_{r}+\lambda_{l}+\theta $, where $ \lambda_{r}+\lambda_{l} $  is a representation of $ L $ on $V$  and $\theta  $ is a $2$-cocycle  of $L$ on $V$.   We called the previous  extension, a standard split extension of $L$ by $V$.                   
\end{theorem}
\begin{defn}
	Let $ ( V,\lambda_{r}+\lambda_{l} ,\alpha_V,\beta_{V}) $ be a representation  of  a  $ L $ by $ V $ and $d=\delta+\lambda_{r}+\lambda_{l}$.  The extension $
	(L\oplus V,d,\alpha+\alpha_V,\beta+\beta_{V})$ is called the 
	semidirectsum of the BiHom-Lie algebras $L$  and $V$.                                                      
\end{defn}

\section{Split extensions of BiHom-Lie algebras}

In this section, we show that any split extension $M$ of $L$ by $V$  is equivalent to $L\oplus V$.
\begin{theorem}\label{extDec} 
	Let
	\[E:0\longrightarrow (V,\mu,\alpha_{V},\beta_{V})\stackrel{i}{\longrightarrow} (M,d',\alpha_{M},\beta_{M})\stackrel{\pi}{\longrightarrow }(L,\delta,\alpha,\beta) \longrightarrow 0,\]   
	be a split extension of $ L $ by $ V $. Then   there exist bilinear map
	$ d=\delta+\lambda_l+\lambda_{r}+\theta+\mu  $ such that $ \lambda_r+\lambda_{l} $ is a representation of $L$,  $\theta$ is a $2$-cocycle and    
	the  split extension  
		\[E_{0}:0\longrightarrow (V,\mu,\alpha_{V},\beta_{M})\stackrel{i_{0}}{\longrightarrow} (L\oplus V,d,\alpha+\alpha_{V},\beta+\beta_{V})\stackrel{\pi_{0}}{\longrightarrow }(L,\delta,\alpha,\beta) \longrightarrow 0,\]	
of $L$ by $V$ is equivalent to $E$. 
\end{theorem}
\begin{proof}
	Let 	\[0\longrightarrow (V,\mu,\alpha_{V},\beta_{V})\stackrel{i}{\longrightarrow} (M,d',\alpha_{M},\beta_{M})\stackrel{\pi}{\longrightarrow }(L,\delta,\alpha,\beta) \longrightarrow 0,\]	
	be a split extension of $ L $. Then there exist 
	a BiHom-Lie-subalgebra $ H\subset M$
	complementary to $ \ker \pi $. Since $Im\, i = \ker \pi$, we have    	
	$M= H\oplus i(V) $.   	   	
	The map  $\pi_{/H}:H\rightarrow L$ (resp $k:V\rightarrow i(V) $) defined by $\pi_{/H}(x)=\pi(x) $ (resp. $k(v)=i(v)$) is bijective, its inverse $ s  $ (resp. $l$).  Considering the map $\Phi\colon L\oplus V \to M$  defined  by
	$ \Phi(x+v)= s(x)+i (  v),$    it is easy to verify that $\Phi$ is an  isomorphism, $ i(V) $ is an ideal of $M$ and $ M=s(L)\oplus i(V) $. Then 
	, $d'=\delta'+\lambda_{r}'+\lambda_{l}'+\theta'+\mu'$ where 
	$ \lambda'=\lambda_{r}'+\lambda_{l}' $  is a representation of $ s(L) $ on $i(V)$ 
	 and $\theta  $ is a $2$-cocycle  of $s(L)$ on $i(V)$.


	Define a bilinear map $ d\colon L\oplus V\to L\oplus V $ by \[d(x+v,y+w) =\Phi^{-1}(d'(s(x)+i(v),s(y)+i(w))) .	\]
	Then
	\begin{equation}\label{ext2021}
		d'\left( \Phi(x+v),\Phi(y+w)\right)= d'(s(x)+i(v),s(y)+i(w))=\Phi\left(d(x+v,y+w) \right)          
	\end{equation} 
	
	Clearly $ \left(L\oplus V,d,\alpha+\alpha_V,\beta+\beta_{V} \right)  $ is a BiHom-Lie algebra and 
	\[E_0:0\longrightarrow (V,\mu,\alpha_{V},\beta_{V})\stackrel{i_{0}}{\longrightarrow} (L\oplus V,d,\alpha+\alpha_{V},\beta+\beta_{V})\stackrel{\pi_{0}}{\longrightarrow }(L,\delta,\alpha,\beta) \longrightarrow 0,\]	 be a split extension of $L$ by $V$
	equivalent to $E$. By \eqref{ext2021} we have 
	\begin{gather*} 
		\delta(x,y)=\pi\left( \delta' (s(x),s(y))\right) ,\, \theta(x,y)=k\left( \theta'(s(x),s(y)) \right),\\ 
	  \lambda_{l}(x,v)=k\left( \lambda_{l}'(s(x),i(v))\right) 
		,\,  \lambda_{r}(v,x)=k\left(  \lambda_{r}'(i(v),s(x))\right) 
		,\,  \mu(v,w)=k\left( ( \mu (i(v),i(w)))\right) 
	\end{gather*}
	and $ d=\delta+\theta+ \lambda_{l}+\lambda_{r}+\mu $.	
\end{proof}
We deduce from the last  theorem that any split extension of $M$ by $V$ is equivalent to 
\[0\longrightarrow (V,\mu,\alpha_{V},\beta_{V})\stackrel{i_{0}}{\longrightarrow} (L\oplus V,d,\alpha+\alpha_{V},\beta+\beta_{V})\stackrel{\pi_{0}}{\longrightarrow }(L,\delta,\alpha,\beta) \longrightarrow 0,\]
where $ i_{0}(v)=v $, $ \pi_{0}(x+v)=x $ and $d=\delta+\theta+ \lambda_{l}+\lambda_{r}+\mu   $. 
\section{Second group of cohomologie and abelian split extensions of BiHom-Lie algebras}
In this section, we search for the equivalence condition of two standard split   abelian extensions.
Define the following standard split extension by 
\[E_{0}:0\longrightarrow (V,\alpha_{V},\beta_{V})\stackrel{i_{0}}{\longrightarrow} (L\oplus V,d,\alpha+\alpha_{V},\beta+\beta_{V})\stackrel{\pi_{0}}{\longrightarrow }(L,\delta,\alpha,\beta) \longrightarrow 0,\]
\[E'_{0}:0\longrightarrow \left( V',\alpha_{V'},\beta_{V'}\right) \stackrel{i'_{0}}{\longrightarrow} (L'\oplus V',d',\alpha'+\alpha_{V'},\beta'+\beta_{V'})\stackrel{\pi'_{0}}{\longrightarrow }(L',\delta',\alpha') \longrightarrow 0,\]
where $ i_{0}(v)=v $, $ \pi_{0}(x+v)=x $. $ i'_{0}(v')=v' $, $ \pi'_{0}(x'+v')=x' $.\\
We assume that $ E_{0} $ and $ E'_{0} $ are equivalents.
Then there exist an isomorphism of BiHom-Lie algebra 
\[\Phi\colon  \left( L\oplus V,d,\alpha+\alpha_V,\beta+\beta_{V}\right)  \to \left( L'\oplus V',d',\alpha'+\alpha_V',\beta'+\beta_{V'}\right)  \]
Satisfies $ \Phi\circ i_0=i'\circ \Phi  $ and $ \pi' \circ \Phi =\Phi \circ \pi_0$.\\
We set $\Phi=s+i $ and $ s=s_1+i_1 $, $i=s_2+i_2$, where  
$s\colon L\to L'\oplus V'   $,
$s_{1}\colon L\to L',   $
$i_{1}\colon L\to V',   $
$i\colon V\to L'\oplus V' ,   $
$s_{2}\colon V\to L',   $
$i_{2}\colon V\to V'.$
We denote $\Phi(v)=v'  $ and $\Phi(x)=x'  $.\\
We have $i(v)=\Phi(v)=\Phi(i_0(v))=i'(\Phi(v))=i'(v')=v'  $\\
and $ x'=\Phi(x)  =\Phi(\pi_0(x))=\pi'(\Phi(x))=\pi'(s(x))=\pi'(s_1(x)+i_1(x))=s_1(x) $.\\
Hence $ i_1(x)=s(x)-x' $, $ i(v)=i_2(v)=v' $, $ s_2=0 $.
We have 
\begin{gather*}
d' (\Phi(x+v),\Phi(y+w))=	d' (x'+v',y'+w')\\
	=d'(s_1(x)+i_1(x)+i(v),s_1(y)+i_1(y)+i(w))\\
	=\delta'(s_{1}(x),s_{1}(y))+\theta'(s_{1}(x),s_{1}(y))
+\lambda'_l (s_{1}(x),i_1(y)+i(w))+\lambda'_r (i_1(x)+i(v),s_1(y))
\end{gather*}
Moreover, 
\begin{align*}
	&d'(\Phi(x+v),\Phi(y+w))\\
	&= \Phi(d(x+v,y+w))\\
	&= \Phi(\delta(x,y)+\theta(x,y))+\lambda_l(x,w)+\lambda_r (v,y)+\mu(v,w)\\	
	&=s_1\circ \delta(x,y)+i_1\circ \delta(x,y)+i_2\circ \theta(x,y)+i_2\circ \lambda_l(x,w)\\
&	+i_2\circ \lambda_r(v,y).
\end{align*}
So,
\begin{align}
s_1\circ \delta(x,y)&=\delta'(s_{1}(x),s_{1}(y))\\
i_1\circ \delta(x,y)+i_2\circ \theta(x,y)&=  \theta'(s_{1}(x),s_{1}(y))
+\lambda'_l (s_{1}(x),i_1(y))\nonumber\\
+&\lambda'_r (i_1(x),s_1(y))
\label{cobord26nov}\\
i\circ \lambda_l(x,w)&=\lambda'_l (s_{1}(x),i(w))
\\
	i\circ \lambda_r(v,y)&=\lambda'_r (i(v),s_1(y))     
\end{align}
\begin{defn} 
We call $1$-coboundary operator of BiHom-Lie algebra $L$ the map $D_{1}\colon C^{1}(L,V)\to C^{2}(L,V)$ defined by
\[D_{1}(f)(x,y)=-f\left( \delta(x,y)\right) +\lambda_l (x,f(y))+\lambda_r (f(x),y).     \]	
\end{defn}	

Define $h\colon L'\to V'$ by $h(s_1(x))=i_1(x)$ then
\[D_{1}(h)(s_1(x),s_1(y))=-h\left( \delta'(s_1(x),s_1(y))\right) +\lambda'_l (s_{1}(x),f(s_1(y)))+\lambda'_r (fs_1(x),s_1(y)).\]
Hence, by \eqref{cobord26nov}, we obtain
\begin{gather}
i_2\circ \theta(x,y)=\theta'(s_{1}(x),s_{1}(y))+D_{1}(h)	(s_1(x),s_1(y)).\label{equivalent}                                  
\end{gather}
\begin{thm}
	With the above notations,
the map $D_{1}(h)$ is a $2$ cocycle  of $L'$ on $V'$ with respect the representation $ \lambda'_r+\lambda'_l $ i.e 
\[ D_2\circ D_1=0\] 	
\end{thm}
\begin{proof}
We have 
\begin{align*}
	 D_{2}\circ D_{1}(h)&=[\delta+\lambda_{r}+\lambda_{l},D_{1}(h)]\\
	&=\delta\circ D_{1}(h) +\lambda_{r}\circ D_{1}(h) +\lambda_{l}\circ D_{1}(h) \\+&  D_{1}(h) \circ\delta+D_{1}(h)\circ \lambda_{r}+D_{1}(h)\circ \lambda_{l}                                    
\end{align*}
Since $D_{1}(h)\in C^{2}(L',V') $, we have $ \delta\circ D_{1}(h)=0 $, $ \lambda_{r}\circ D_{1}(h)=0 $, $ D_{1}(h)\circ \lambda_{r}=0 $ and $ D_{1}(h)\circ \lambda_{l}=0 $. Then, $	 D_{2}\circ D_{1}(h)=
\lambda_{l}\circ D_{1}(h) +D_{1}(h)\circ \delta    $. Therefore,
\begin{align}
 &D_{2}\circ D_{1}(h)(x',y',z')\nonumber\\	
&=\lambda_{l}\left(\beta^{2}(x'),D_{1}(h)(\beta(y'),\alpha(z')) \right)\label{a17dec2021}\\
-&\lambda_{l}\left(\beta^{2}(y'),D_{1}(h)(\beta(x'),\alpha(z')) \right)\label{b17dec2021}\\
+&\lambda_{l}\left(\beta^{2}(z'),D_{1}(h)(\beta(x'),\alpha(y')) \right)\label{c17dec2021}	
\end{align}
\begin{gather}
+D_{1}(h)	\left(\beta^{2}(x),\delta(\beta(y),\alpha(z)) \right)\label{d17dec2021}\\
-D_{1}(h)\left(\beta^{2}(y'),\delta(\beta(x'),\alpha(z')) \right)\label{e17dec2021}\\
+D_{1}(h)\left(\beta^{2}(z'),\delta(\beta(x'),\alpha(y')) \right)\label{f17dec2021}
\end{gather}
Using $h\circ \beta=\beta\circ h$, we have 
\begin{gather}
	\eqref{a17dec2021}= 
		-\lambda_{l}\Big(\beta^{2}(x'), h\delta'(\beta(y'),\alpha(z')) \Big)\label{28nov}\\
		+\lambda_{l}\Big(\beta^{2}(x'),\lambda'_l (\beta(y'),\alpha(h(z')))  	\Big)\label{jaco3Nov28}\\
		+\lambda_{l}\Big(\beta^{2}(x'),\lambda'_r ( \beta(h(y')),\alpha(z'))	\Big)\label{di27Nov1}
\end{gather}
\begin{gather}
	\eqref{b17dec2021}= 
		\lambda_{l}(\beta^{2}(y'),h\delta'(\beta(x'),\alpha(z')))\label{dim28Nov}\\
		-\lambda_{l}(\beta^{2}(y'),\lambda'_l (\beta(x'),\alpha(h(z'))	))\label{jaco2Nov28}\\
		-\lambda_{l}(\beta^{2}(y'),\lambda'_r (\beta(h(x')),\alpha(z'))	)\label{r2Nov27b}
\end{gather}
\begin{gather}
	\eqref{c17dec2021}= 	-\lambda_{l}\Big(\beta^{2}(z'),h\delta'(\beta(x'),\alpha(y'))\Big)\label{2dim28nov}\\
	+\lambda_{l}\Big(\beta^{2}(z'),\lambda'_l (\beta(x'),\alpha(h(y')))\Big)\label{r3Nov28c}\\
	+\lambda_{l}\Big(\beta^{2}(z'),\lambda'_r (\beta(h(x')),\alpha(y'))\Big)\label{r3Nov27c}
\end{gather}
\begin{gather}
	\eqref{d17dec2021}=-h\delta'(\beta^{2}(x'),\delta(\beta(y'),\alpha(z')))\label{Nov27a}\\
	+\lambda'_l (\beta^{2}(x),h(\delta(\beta(y),\alpha(z)) ))\label{Sam27Nov2}\\
	+\lambda'_r (\beta^{2}(h(x')),\delta(\beta(y),\alpha(z)) )	\label{r1Nov27a} 
\end{gather}
\begin{gather}
	\eqref{e17dec2021}=h\delta'(\beta^{2}(y'), \delta(\beta(x'),\alpha(z'))    )\label{Nov27b}\\
	-\lambda'_l (\beta^{2}(y'),h(\delta(\beta(x'),\alpha(z'))))\label{Nov28a}\\
	-\lambda'_r (\beta^{2}(h(y')),s_1(\delta(\beta(x'),\alpha(z')) ))\label{repNov28a}
\end{gather}
\begin{gather}
	\eqref{f17dec2021}=-h\delta'(\beta^{2}(z'),\delta(\beta(x'),\alpha(y')))\label{Nov27c}\\
	+\lambda'_l (\beta^{2}(z'),h(\delta(\beta(x'),\alpha(y'))))\label{1Nov27c}\\
	+\lambda'_r (\beta^{2}(h(z')),\delta(\beta(x'),\alpha(y')))\label{jaco1Nov28}
\end{gather}
Since $\lambda'_r +	\lambda'_l$ is a representation of $L'$ on $V'$, we have 
\[\eqref{r1Nov27a}+\eqref{r2Nov27b}+\eqref{r3Nov27c}=0\]
and
\[ \eqref{repNov28a}+\eqref{di27Nov1}+\eqref{r3Nov28c}=0.  \]
By the BiHom-Jacobi identity, we obtain that
\[\eqref{Nov27a}+\eqref{Nov27b}+\eqref{Nov27c}=0  \]
and 
\[ \eqref{jaco1Nov28}+\eqref{jaco3Nov28}+\eqref{jaco2Nov28}=0. \]
it's clear that $\eqref{28nov}+\eqref{Sam27Nov2}=0$ and $\eqref{1Nov27c}+\eqref{2dim28nov}=0$. This completes the proof. 		
\end{proof}
\begin{defn}
The space of $2$-coboundaries of $L$ on $V$ is 
\[ B^{2}(L,V)=Im(D_1)=\{D_1(f)\mid f\in C^1 (L,V)  \}. \]	
\end{defn}	
\begin{defn}
We call the $2^{th}$ cohomology group of the BiHom-Lie algebra $L$, the quotient	
	    \[ H^2(L,V,\lambda_{r}+\lambda_{r}) =Z^{2}(L,V,\lambda_{r}+\lambda_{r})/ B^{2}(L,V,\lambda_{r}+\lambda_{r})=\{\bar{\theta}\mid \theta\in Z^{2}(L,V,\lambda_{r}+\lambda_{r})\} . \]`
	    
\end{defn}

\begin{lem}

Let two equivalents standard split abelian extension
\begin{gather*}
	\xymatrix {\relax 0 \ar[r] &( V,\alpha_V,\beta_V )   \ar[d]_{\varphi}\relax\ar[r]^-{i_{0}} \relax & (L\oplus V,d,\alpha+\alpha_V,\beta+\beta_V ) \ar[d]_{\Phi} \ar[r]^-{\pi_{0}}& (L,\delta,\alpha,\beta)\ar[r] \ar[d]_{s}&0 \\
	\relax	0\ar[r] & ( V',\alpha_V',\beta_V' ) \ar[r]^-{i_{0}'}     & (L\oplus V,d,\alpha+\alpha_V,\beta+\beta_V ) \ar[r]^-{\pi_{0}'} &(L',\delta',\alpha',\beta')\ar[r]&0}
\end{gather*}
where $d=\delta+\lambda_l+\lambda_{r}+\theta  $ and $d'=\delta'+\lambda_l'+\lambda_{r}'+\theta'  $. Then, $i\circ \theta$ is a $2$-cocycle of $L'$ on $V'$ and it is equivalent to $\theta'$ (i.e $\overline{ i\circ \theta}=\bar{\theta'} $).
\end{lem}
\begin{proof}
	 By \eqref{equivalent}, we have $i_2\circ \theta(x,y)-\theta'(s_{1}(x),s_{1}(y))\in B^{2}(L',V',\lambda_{r}+\lambda_{r})$. Hence $  \overline{ i\circ \theta}=\bar{\theta'} .$                           
\end{proof}
\begin{lem}
If $\theta$ and $\theta'$ are two equivalent $2$ cocycle of $L$ on $V$. Then, 
the extensions $(L\oplus V,d,\alpha+\alpha_V,\beta+\beta_V )$ and $(L\oplus V,d',\alpha+\alpha_V,\beta+\beta_V )$ are equivalent.
\end{lem}
\begin{proof}
Let $\theta$ and $\theta'$ are two equivalent $2$ cocycle of $L$ on $V$. Then, there exist $h\in C^1(L,V)$ satisfies $\theta'=\theta+D_1(h)$. Define the linear map $\Phi\colon L\oplus V\to L\oplus V$ by $\Phi(x+v)=x-h(x)+v$. Clearly, $\Phi$ is an isomorphism. Moreover, we have
\begin{gather*}
d'\left(\Phi(x+v),\Phi(y+w) \right)=d'\left(x-h(x)+v,y-h(y)+w) \right) \\
=\delta(x,y) +\lambda_{l}(x,w)+\lambda_{r}(v,y)-	\lambda_{l}(x,h(y))-
\lambda_{r}(h(x),y)+\theta'(x,y)\\
=\delta(x,y) +\lambda_{l}(x,w)+\lambda_{r}(v,y)-	\lambda_{l}(x,h(y))-
\lambda_{r}(h(x),y)+\theta(x,y)+D_1(h)(x,y)\\
=\delta(x,y) +\lambda_{l}(x,w)+\lambda_{r}(v,y)+\theta(x,y)-h\left(\delta(x,y) \right)\\
=\Phi(d(x+v,y+w)).
\end{gather*}
Hence, $\Phi\colon L\oplus V\to L\oplus V$ is an isomorphism of BiHom-Lie algebra. Therefore,                                  
 the following diagram 
\begin{gather*}
	\xymatrix {\relax 0 \ar[r] &( V,\alpha_V,\beta_V )   \ar[d]_{id_V}\relax\ar[r]^-{i_{0}} \relax & (L\oplus V,d,\alpha+\alpha_V,\beta+\beta_V ) \ar[d]_{\Phi} \ar[r]^-{\pi_{0}}& (L,\delta,\alpha,\beta)\ar[r] \ar[d]_{id_L}&0 \\
		\relax	0\ar[r] & ( V,\alpha_V,\beta_V ) \ar[r]^-{i_{0}}     & (L\oplus V,d',\alpha+\alpha_V,\beta+\beta_V ) \ar[r]^-{\pi_{0}} &(L,\delta,\alpha,\beta)\ar[r]&0}
\end{gather*}
gives, the extensions $(L\oplus V,d,\alpha+\alpha_V,\beta+\beta_V )$ and $(L\oplus V,d',\alpha+\alpha_V,\beta+\beta_V )$ are equivalent.	
	\begin{thm}
		The set $ Ext(L,V,\lambda_{r}+\lambda_{l}) $  of equivalence classes of  split extensions of $ \left(L,\delta, \alpha,\beta\right)  $  by an abelian BiHom-Lie algebra  $(V,\alpha_{V},\beta_{V})$
		is one-to-one   correspondence with $ Z^2(L,V,\lambda_{r}+\lambda_{l})/ B^2(L,V,\lambda_{r}+\lambda_{l})$, that is 
		\[ Ext(L,V,\lambda_{r}+\lambda_{l})\cong  Z^2(L,V,\lambda_{r}+\lambda_{l})/ B^2(L,V,\lambda_{r}+\lambda_{l}). \]	                     
	\end{thm}

\end{proof}

\end{document}